\documentclass[11pt,reqno]{amsart}
\usepackage{color}

\usepackage{amsmath}
\usepackage{amsfonts,amscd}
\usepackage{amssymb}
\usepackage{url}
\usepackage{hyperref}

\usepackage[english]{babel}
\usepackage{booktabs}
\usepackage{tikz}

\theoremstyle{plain}
\newtheorem{theorem}                 {Theorem}[section]

\newtheorem{lemma}        [theorem]  {Lemma}
\newtheorem{proposition}  [theorem]  {Proposition}

\theoremstyle{definition}

\newtheorem{example}      [theorem]  {Example}

\newtheorem{definition}   [theorem]  {Definition}

\newtheorem{remark}       [theorem]  {Remark}

\numberwithin{equation}{section}

\def \rn{{\mathbb R}}

\def \F{\mathcal F}
\def \H{\mathcal H}

\def \V{\mathcal V}

\def \nab#1#2{\hbox{$\nabla$\kern -.3em\lower 1.0 ex
		\hbox{$#1$}\kern -.1 em {$#2$}}}

\def \nab#1#2{\nabla_{#1}{#2}}

\def \lb#1#2{[#1,#2]}

\def \g{\mathfrak{g}}
\def \h{\mathfrak{h}}
\def \k{\mathfrak{k}}

\def \m{\mathfrak{m}}

\def \p{\mathfrak{p}}

\def \r{\mathfrak{r}}

\DeclareMathOperator{\ad}{ad}

\def \SLR#1{\text{\bf SL}_{#1}(\rn)}
\def \SL2{\widetilde{\text{\bf SL}}_{2}(\rn)}
\def \slr#1{\mathfrak{sl}_{#1}(\rn)}

\def \SU#1{\text{\bf SU}(#1)}
\def \su#1{\mathfrak{su}(#1)}

\def \sol{\mathfrak{sol}}

\DeclareMathOperator{\trace}{trace}

\numberwithin{equation}{section}

\allowdisplaybreaks

\begin{document}

\subjclass[2020]{53C35, 53C43, 58E20}
	
\keywords{harmonic morphisms, Lie groups, conformal and minimal foliations}
	
\author{Sigmundur Gudmundsson}
\address{Mathematics, Faculty of Science\\
University of Lund\\
Box 118, Lund 221 00\\
Sweden}
\email{Sigmundur.Gudmundsson@math.lu.se}
	
\author{Thomas Jack Munn}
\address{Mathematics, Faculty of Science\\
University of Lund\\
Box 118, Lund 221 00\\
Sweden}
\email{Thomas.Munn@math.lu.se}

\title
[Harmonic Morphisms and Minimal Conformal Lie Foliations]
{Harmonic Morphisms and Minimal\\ Conformal Foliations on Lie Groups}

\begin{abstract}
Let $G$ be a Lie group equipped with a left-invariant Riemannian metric. Let $K$ be a semisimple and normal subgroup of $G$ generating a left-invariant conformal foliation $\F$ on $G$.  We then show that the foliation $\F$ is Riemannian and minimal.  This means that locally the leaves of $\F$ are fibres of a harmonic morphism. We also prove that if the metric restricted to $K$ is biinvariant then $\F$ is totally geodesic.
\end{abstract}
	
	
\maketitle

\section{Introduction and Main Results}
\label{section-introduction}

In differential geometry the study of minimal submanifolds of a given Riemannian ambient space $(M,g)$ plays a central role. Harmonic morphisms $\phi:(M,g)\to (N,h)$ between Riemannian manifolds are useful tools for the construction of minimal submanifolds of a given ambient space $M$.
\smallskip 

Harmonic morphisms are the solutions to an  over-determined,  non-linear system of partial differential equations.  For this reason they are difficult to find, and there exist 3-dimensional Riemannian Lie groups which do not allow for any solutions, not even locally, see \cite {Gud-Sve-5}.  For the general theory of harmonic morphisms between Riemannian manifolds we refer to the excellent book \cite{Bai-Woo-book} and the regularly updated online bibliography \cite{Gud-bib}.
\smallskip

The relationship between minimal foliations and harmonic morphisms has been well studied.  The following two results can be seen as the most important ones.

\begin{theorem}\cite{Wood1986}
A foliation on a Riemannian manifold, of codimension $n=2$, produces harmonic morphisms if and only if it is conformal and has minimal leaves.
\end{theorem}

\begin{theorem}\cite{Bry}\label{theorem-homothetic} 
A homothetic foliation on a Riemannian manifold, of codimension $n\neq 2$, produces harmonic morphisms if and only if the mean curvature vector $\mu^\mathcal{V}$ is of gradient type, or equivalently, if its dual $(\mu^\V)^\flat$ is closed.
\end{theorem}
\begin{remark}
It is then clear that if a homothetic foliation of codimension $n \neq 2$ has minimal leaves then it produces harmonic morphisms. Lemma \ref{lemma-HM-need-minimal-fibres} shows that in the setting of this work, the converse holds too.
\end{remark}

The case of conformal left-invariant foliations of Riemannian Lie groups by semisimple subgroups of codimension two is examined in \cite{Gud-Mun-2}. The main results presented here are generalisations of those in \cite{Gud-Mun-2} to arbitrary codimension. In fact if $K$ is a codimension two semisimple subgroup of $G$ that generates a left-invariant conformal foliation then $K$ must be a normal subgroup of $G$, see \cite{Gud-Sve-6}, this is not true for higher codimension cases and therefore we need to impose the additional condition that $K$ be a \emph{normal} subgroup of G.
The following example shows that we can not extend the results of \cite{Gud-Mun-2} without this additional condition.
\smallbreak

\begin{example}
Let $(G,g)$ be a six-dimensional Riemannian Lie group such that the orthonormal basis $\{A,B,C,X,Y,Z\}$ for its Lie algebra $\g$ satisfies the following Lie bracket relations

$$\lb AB = 2C,\quad \lb CA = 2B,\quad \lb BC = 2A,$$

\begin{eqnarray*}
&&\lb AX = -B+2Z,\quad \lb AY= 0,\quad\lb AZ = A+C-2X,\\
&&\lb BX = A-C,\quad \lb BY = 2Z,\quad\lb BZ = -2Y,\\
&&\lb CX = B+2Y,\quad\lb CY = A+C-2X,\quad \lb CZ = 0,\\
&&\lb XY =  \tfrac{1}{2}\,A - X,\quad \lb XZ = \tfrac{1}{2}\,C - X,\quad \lb YZ = -\tfrac{1}{2}\,B.
\end{eqnarray*}
Here the compact semisimple $K=\SU 2$ is equipped with a biinvariant Riemannian metric (Example \ref{example-9}).  It is \emph{not normal} in $G$ since its Lie algebra $\k$, generated by $A,B,C$ is not an ideal in $\g$. The left-invariant foliation $\F$ generated by $K$ is Riemannian but it is \emph{not minimal}.
\end{example}

We now state the three most important results of this work.

\begin{theorem}
\label{theorem-totally-geodesic}
Let $(G,g)$ be a Riemannian Lie group and $K$ be a semisimple normal subgroup of $G$ generating a left-invariant conformal foliation $\F$ on $G$. If the restriction of the metric $g$ to $K$ is biinvariant then the foliation $\F$ has totally geodesic leaves.
\end{theorem}

\begin{theorem}
\label{theorem-minimal}
Let $(G,g)$ be a Riemannian Lie group and $K$ be a semisimple subgroup of $G$ generating a left-invariant conformal foliation $\F$ on $G$.  If $K$ is normal in $G$ then $\F$ is Riemannian with minimal leaves.
\end{theorem}

As a consequence of Theorem \ref{theorem-minimal} we have the following result concerning harmonic morphisms.

\begin{theorem}
\label{theorem-harmonic-morphism}
Let $(G,g)$ be a Riemannian Lie group and $K$ be a semisimple normal subgroup of $G$ generating a left-invariant foliation $\F$ on $G$. Let $\phi:G\to G/K$ be the natural projection onto the quotient group.  If $G/K$ is equipped with its unique Riemannian metric $h$ such that $\phi:(G,g)\to (G/K,h)$ is a Riemannian submersion then $\phi$ is a harmonic morphism.
\end{theorem}

The additional assumption that $K$ is normal is quite strong, but in the case when $K$ is of codimension three, Proposition \ref{proposition-not-perfect} shows that it is enough to require that the ambient group $G$ is not perfect.

\section{Conformal Foliations}
\label{section-conformal-foliations}

Let $(M,g)$ be a Riemannian manifold, $\V$ be an integrable distribution on $M$ and $\H$ be its orthogonal complementary distribution.  As customary, we will also denote by $\V,\H$ the orthogonal projections onto the corresponding subbundles of $TM$, and by $\F$ the foliation tangent to $\V$. Then the {\it second fundamental form} $B^\V$ of $\V$ is given by 
$$B^{\V}(V,W) = \tfrac{1}{2}\cdot\H(\nab VW + \nab WV), \ \ \text{ where } V,W\in\V.$$
The corresponding {\it second fundamental form} $B^\H$ of $\H$ satisfies
$$B^{\H}(X,Y) = \tfrac{1}{2} \cdot \V(\nab XY + \nab YX), \ \ \text{ where } X,Y \in \H.$$
		
The foliation $\F$ tangent to $\V$ is said to be {\it conformal} if there exists a vertical  vector field $V\in\V$ such that
$$B^{\H}=g\otimes V$$
and $\F$ is said to be {\it Riemannian} if $V =0$. Furthermore, $\F$ is {\it minimal} if $\trace B^{\V} =0$ and {\it totally geodesic} if $B^{\V}=0$. This is equivalent to the leaves of $\F$ being minimal and totally geodesic submanifolds of $M$, respectively.
\smallskip

The following is an immediate consequence of the above definition.

\begin{proposition}
\label{prop-general-conformality}
\cite{Bai-Woo-book}
Let $(M,g)$ be a Riemannian manifold, $\V$ be an integrable distribution on $M$ and $\H$ be its orthogonal complementary distribution. Then the foliation $\F$ tangent to $\V$ is conformal if and only if 
$$B^\H(X,X)=B^\H(Y,Y)\ \ \text{and}\ \ B^\H(X,Y)=0,$$
for any orthonormal horizontal vectors $X,Y\in\H$.  If $\F$ is conformal then it is Riemannian if and only if for all $X,Y\in\H$ we have
$$B^\H(X,X)+ B^\H(Y,Y)=0.$$
\end{proposition}

\section{Conformal Lie Foliations}
\label{section-conformal-Lie-foliations}

Let $(G,g)$ be an $m$-dimensional Riemannian Lie group with a subgroup $K$ generating a left-invariant foliation $\F$ on $G$, of codimension $n$, by left translations.  Further let $\g=\k\oplus \m$ be an orthogonal decomposition of the Lie algebra $\g$ of $G$, where $\k$ is the Lie algebra of $K$. Let $\{ V_1,\dots,V_d,X_1\dots ,X_n\}$ be an orthonormal basis for $\g$ such that $V_1,\dots,V_d$ generate the subalgebra $\k$. Then we can describe the Lie brackets of $\k$ by the following system
$$
\ad_{V_\alpha}(V_\beta)=[V_\alpha,V_\beta]=\sum_{\gamma=1}^d c^\gamma_{\alpha\beta} V_\gamma,
$$
where the $c^\gamma_{\alpha\beta}$ are real structure constants of $\k$.  The rest of the bracket relations for $\g$ satisfy
$$
\ad_{X_i}(V_\alpha)
=[X_i,V_\alpha]
=\sum_{\gamma=1}^d x^\gamma_{i\alpha} V_\gamma
+\sum_{k=1}^n y_{i\alpha}^k\,X_k,
$$ 
$$
\ad_{X_i}(X_j)
=[X_i,X_j]
=\sum_{\gamma=1}^d \theta^\gamma_{ij}\, V_\gamma
+\sum_{k=1}^n\rho_{ij}^k\,X_k,$$
for some real coefficients $x^\gamma_{i\alpha},y_{ij}^k,\rho_{ij}^k,\theta_{ij}^\gamma$.
\smallskip

The Koszul formula for the Levi-Civita connection $\nabla$ on $(G,g)$ implies that the second fundamental forms $B^{\V}$ and $B^{\H}$ of the vertical and horizontal  distributions, respectively, satisfy
\begin{equation*}
B^{\V}(V_\alpha,V_\beta) 
=\tfrac 12\cdot\sum_{k=1}^n\big( (g([X_k,V_\alpha],V_\beta)
+g([X_k,V_\beta],V_\alpha)\big)\cdot X_k ,
\end{equation*}
\begin{equation*}\label{equation-BH}
B^{\H}(X_i,X_j) 
=\tfrac 12\cdot\sum_{\gamma=1}^d\big( g([V_\gamma,X_i],X_j)
+g([V_\gamma,X_j],X_i)\big)\cdot V_\gamma.
\end{equation*}

\begin{remark}
We now fix the notation for the examples given throughout this work. We denote vertical basis vectors by $A,B,C$ and horizontal basis vectors by $X,Y,Z,W$.
\end{remark}
\begin{example}
Let $(G,g)$ be the four-dimensional solvable Riemannian Lie group such that the orthonormal basis $\{A,B,X,Y\}$ for its Lie algebra $\g$ satisfies the following bracket relations
$$[A, B] = -2A,$$
$$
[B, X] = A + X + Y,
\quad [B, Y] = -X + Y,
\quad [X, Y] = A.
$$
The subgroup $K=H^2$ is solvable (Example \ref{example-02}).  The left-invariant foliation $\F$ generated by $K$ is conformal but not Riemannian.  The horizontal distribution $\H$ is not integrable.   
\end{example}

The following result provides us with a sufficient condition for a conformal foliation $\F$ on $G$ to be Riemannian.  It generalises Proposition 3.3 of \cite{Gud-12}. 

\begin{proposition}
\label{proposition-conformal-foliation-is-Riemannian}
Let $(G,g)$ be a Riemannian Lie group with a subgroup $K$ generating a left-invariant conformal foliation $\F$ on $G$. If $K$ is semisimple then the foliation $\F$ is Riemannian.
\end{proposition}

\begin{remark}
Note that Riemannian foliations are homothetic so Theorem \ref{theorem-homothetic} can be applied. 
\end{remark}

\begin{proof}[Proof of Proposition \ref {proposition-conformal-foliation-is-Riemannian}]
Since $\F$ is a conformal foliation, we have that for $V \in \V$, the adjoint action on the horizontal space $\H$ is conformal i.e. for any $X,Y \in \H$ we have
$$ \rho(V) \cdot g(X,Y) = g([V,X],Y)+g([V,Y],X).$$
Notice that for a fixed horizontal unit vector $X_i$ we have that 
$$ \rho(V) = 2\, g([V,X_i],X_i).$$
Then by the linearity of the metric, the Lie bracket relations and the fact that $[\V,\V] = \V$, as $K$ is semisimple, we see that $\rho$ vanishes if and only if for all $1\le\alpha,\beta\le d$ we have 
$$ g([[V_\alpha,V_\beta],X_i],X_i) = 0.$$ 
We now show that this is the case
\begin{eqnarray*}
&&g([[V_\alpha,V_\beta],X_i],X_i)\\
&=&-g([[X_i,V_\alpha],V_\beta]+[[V_\beta,X_i],V_\alpha],X_i) \\
&=&-g( [\sum_{k=1}^ny_{i \alpha}^k \, X_k + 
\sum_{\gamma =1}^d x_{i \alpha}^\gamma\, V_\gamma, V_\beta] 
-[\sum_{k=1}^n y_{i \beta}^k \, X_k 
+\sum_{\gamma=1}^d x_{i \beta}^\gamma\, V_\gamma, V_\alpha],X_i) \\
&=&-\sum_{k=1}^n(y_{i \alpha}^k \cdot y_{k \beta}^i - y_{i \beta}^k \cdot y_{k \alpha}^i) \\
&=&-\sum_{ k \neq i}  (y_{i \alpha}^k \cdot y_{k \beta}^i - y_{i \beta}^k \cdot y_{k \alpha}^i) \\
&=&-\sum_{ k \neq i} ( y_{i \alpha}^k \cdot y_{k \beta}^i - y_{k \beta}^i \cdot y_{i \alpha}^k) \\
&=&0.
\end{eqnarray*}
The last step follows from the fact that the conformality implies that if $k\neq i$ then
$$y_{i \alpha}^k+y_{k \alpha}^i
=g([V_\alpha,X_i],X_k)+g([V_\alpha,X_k],X_i)=\rho(V_\alpha) \cdot g(X_i,X_k)=0.$$
\end{proof}

If we further assume that the semisimple subgroup $K$ is normal we can define the homogeneous quotient space $G/K$.

\begin{remark}\label{remark-O'Neill}
If the foliation $\F$ on the Lie group $G$ is Riemannian then there exists a unique Riemannian metric $h$ on the homogeneous quotient space $G/K$ such that the natural projection $\pi:(G,g)\to (G/K,h)$ is a Riemannian submersion.  The horizontal vector fields $X_1,\dots ,X_n$ are basic and we denote by $\bar X_1,\dots \bar X_n$ the vector fields $d\pi(X_1),\dots ,d\pi(X_n)$ forming a global orthonormal frame of the tangent bundle of the quotient space $G/K$.  According to Lemma 1 on page 460 of O'Neill's fundamental work \cite{ONe}, we have 
$$\bar\nabla_{\bar X_i}{\bar X_j}=d\pi(\nabla_{X_i}{X_j}).$$
Here $\nabla$ and $\bar\nabla$ are the Levi-Civita connections of $(G,g)$ and $(G/K,h)$, respectively.
This relation is useful for determining the geometry of the Riemannian homogeneous quotient space $(G/K,h)$.
\end{remark}

\section{Minimality and Total Geodecity}
\label{section-minimality}
	
Let $(G,g)$ be an $m$-dimensional Riemannian Lie group with a semisimple subgroup $K$, generating a left-invariant conformal foliation $\F$ on $G$, of codimension $n$. Further let $\g=\k\oplus \m$ be an orthogonal decomposition of the Lie algebra $\g$ of $G$, where $\k$ is the Lie algebra of $K$. Let $\{ V_1,\dots,V_d,X_1\dots ,X_n\}$ be an orthonormal basis for $\g$ such that $V_1,\dots,V_d$ generate the subalgebra $\k$. Then we can describe the Lie brackets of $\k$ by the following system

$$
\ad_{X_i}(V_\alpha)
=[X_i,V_\alpha]
=\sum_{\gamma=1}^d x^\gamma_{i\alpha} V_\gamma
+\sum_{k=1}^n y_{i\alpha}^k\,X_k,
$$ 

$$\ad_{X_i}(X_j)=[X_i,X_j]=\sum_{k=1}^n\rho_{ij}^k\,X_k
+\sum_{\gamma=1}^d \theta^\gamma_{ij}\, V_\gamma,$$
for some real coefficients $x^\gamma_{i\alpha},y_{i\alpha}^k,\rho_{ij}^k,\theta_{ij}^\gamma$.
\smallskip

The structure constants of $\g$ can be used to describe when the foliation $\F$ is minimal or even totally geodesic. First we notice that for the second fundamental form $B^{\V}$ of $\V$ we have
\begin{eqnarray*}
2\cdot B^{\V}(V_\alpha,V_\beta) 
&=&\sum_{k=1}^n\big(g([X_k,V_\alpha],V_\beta) +g([X_k,V_\beta],V_\alpha)\big)\cdot X_k\\
&=&\sum_{k=1}^n\sum_{\gamma=1}^d \big(
 g(x_{k\alpha}^\gamma V_\gamma,V_\beta )
+g(x_{k\beta }^\gamma V_\gamma,V_\alpha)\big)
    \cdot X_k\\
   && +\sum_{k=1}^n\sum_{i=1}^n \big(
    g(y_{k\alpha}^i X_i,V_\beta )
    +g(y_{k\beta}^i X_i,V_\alpha)\big)
    \cdot X_k\\
&=&\sum_{k=1}^n
(x_{k\alpha}^\beta+x_{k\beta}^\alpha)\cdot X_k.
\end{eqnarray*}
From this we see that the foliation $\F$ is {\it minimal} if and only if for $1\le i\le n$ we have
$$ \sum_{\alpha=1}^d\, x_{i\alpha}^\alpha = 0.$$
Furthermore, we observe that $\F$ is {\it totally geodesic} if and only if for all $1\le i\le n$ and $1\le \alpha,\beta\le d$ we have 
$$x_{i\alpha}^\beta+x_{i\beta}^\alpha=0.$$

\begin{lemma}\label{lemma-HM-need-minimal-fibres}
Let $(G,g)$ be a Riemannian Lie group and $K$ be a semisimple subgroup of $G$ generating a left-invariant conformal foliation $\F$ on $G$. Then $\F$ produces harmonic morphisms if and only if it has minimal leaves.
\end{lemma}
\begin{proof}
By Proposition \ref{proposition-conformal-foliation-is-Riemannian} $\F$ is Riemannian and thus homothetic.
Since $\F$ is homothetic by Theorem \ref{theorem-homothetic} it produces harmonic morphisms if and only if $d (\mu^\V)^\flat=0$ i.e.
\begin{eqnarray*}
d  (\mu^\V)^\flat (E,F) &=& d_E ( (\mu^\V)^\flat(F)) -d_F ( (\mu^\V)^\flat(E)) - (\mu^\V)^\flat ([E,F]) \\
&=& E(\langle \mu^\V, F\rangle) - F(\langle \mu^\V, E\rangle) - \langle \mu^\V, [E,F] \rangle  \\
&=& 0,
\end{eqnarray*}
for all left-invariant vector fields $E,F \in \g$. In particular, since the mean curvature vector $$\mu^\V = \sum_{k=1}^n \sum_{\alpha =1}^d x_{i\alpha}^\alpha X_i$$ is left-invariant $\F$ produces harmonic morphisms if and only if 

$$ \langle \mu^\V, [E,F] \rangle  =0, $$
for all $E,F \in \mathfrak{g}$.

We now suppose that $\mu^\V \neq 0$ but $\mu^\V \perp [\g,\g]$ and will obtain a contradiction.

First choose an orthonormal basis $\{\hat X_1 , \dots \hat X_n\}$ such that $\hat X_1$ is the normalised vector $\mu^\V$ i.e.

$$ \mu^\V = \sum_{\alpha =1}^d x_{1 \alpha}^\alpha \hat X_1. $$

Now we recall the Levi decomposition 
$ \g = \mathfrak{s} \oplus \mathfrak{r} $,
where $\mathfrak{s}$ is semisimple and $\mathfrak{r}$ is the radical of $\g$. Then $\k \subset \mathfrak{s}$ because $\k \cap \mathfrak{r}$ is a solvable ideal of $\k$ (since $\mathfrak{r}$ is a radical), but since $\k$ is semisimple, the only such ideal is the trivial one $\{0\}$.

Furthermore $\mu^\V \in \mathfrak{r}$ since $$ \mu^\V\perp [\g, \g] \supset [\mathfrak{s}, \mathfrak{s}] = \mathfrak{s}.$$
So $[\mu^\V, \V] \subset \mathfrak{r}$, in particular 
$$ \V[ \mu^\V, \V ] = 0$$
and thus 
$$\mu^\V = \sum_{\alpha =1}^d  \big( g([\hat X_1, V_\alpha], V_\alpha) \big) \cdot \hat X_1 =0. $$

\end{proof}

\section{The Proofs of our Main Results}
This section begins by describing standard left-invariant Riemannian metrics on Lie groups.
For the details from Lie theory, we refer to the standard work \cite{Hel} of S. Helgason.
\smallskip

Let $K$ be a real semisimple Lie group with Lie algebra $\k$.  Let $B:\k\times\k\to\rn$ be its symmetric, bilinear and non-degenerate Killing form with  $$B(V,W)=\trace(\ad_V \circ \ad_W).$$
Let $\tilde K$ be the maximal compact subgroup of $K$ with Lie algebra $\tilde\k$.  Then we have an orthogonal decomposition $\k=\tilde\k\oplus\p$ of $\k$ with respect to the Killing form $B$, which is negative definite on $\tilde\k$ and positive definite  on $\p$.
\smallskip

In the following proofs we assume that $K$ is a \emph{normal} subgroup of $G$, ensuring that $\k$ is an ideal of $\g$ i.e. the structure constants $y_{i\alpha}^k = 0$ for all $i,k$ and  $\alpha$. So we may write

$$
\ad_{X_i}(V_\alpha)
=[X_i,V_\alpha]
=\sum_{\gamma=1}^d x^\gamma_{i\alpha} V_\gamma.
$$ 
We impose this condition because in general $\V$ and $\H$ may not be orthogonal with respect to the Killing form, so the arguments in the proofs of Theorem \ref{theorem-totally-geodesic} and Proposition \label{Proposition-CK-minimal} would fail without it.
\smallskip

\begin{proof}[Proof of Theorem \ref{theorem-totally-geodesic}]
We write the decomposition of the semisimple subgroup $K$ into simple factors
$$ K = K_1 \times \cdots \times K_r$$
in terms of Lie algebras we have that $\mathfrak{k} = \mathfrak{k}_1 \oplus \dots \oplus \mathfrak{k}_r$, where each $\mathfrak{k_t}$ is a \emph{simple} Lie ideal of $\k$. Since $\k$ is a semisimple, every derivation of $\k$ will be an inner derivation, in particular since $\k$ is an ideal of $\g$, we have that $ad_g : \k \to \k$ must be an inner derivation, i.e. for $g \in \g$ there exists $k_0 \in \k$ such that
$$ [g,\k] = [k_0,\k].$$
Thus it follows that each $\k_t$ will also be an ideal of $\g$ since 
$$ [g,\k_t] = [k_0,\k_t] \subset \k_t$$ 
so when using an orthonormal basis that is compatible with the decomposition of $\k$ we have that $x_{i \beta}^\alpha = 0 $ for any $V_\alpha, V_\beta$ in different simple components $\k_t,\k_s$ respectively. It remains to consider the case when $V_\alpha, V_\beta$ are contained in the same $\k_t$.

Since the restriction of the left-invariant Riemannian metric $g$ on $G$ to $K$ is biinvariant it is a negative multiple of the Killing form $B:\g\times\g\to\rn$ on each simple component $K_t$ of $K$ i.e. for some $\Lambda_1,\dots \Lambda_r \in\rn^+$ we have 
$$g|_{\k_t \times \k_t} = -\Lambda_t\cdot B|_{\k_t \times \k_t}.$$
Further, let $\{V_1,\dots V_d\}$ be an orthonormal basis for the semisimple subalgebra $\k$ of $\g$ with respect to $g$ i.e.  $$
g(V_\alpha,V_\beta)=- \Lambda_t \cdot B(V_\alpha,V_\beta)=\delta_{\alpha\beta},
$$
for all $1\le \alpha ,\beta\le n$.
 Then
\begin{eqnarray*}
B([V_\alpha,V_\beta],X_i) 
&=& B(V_\alpha,[V_\beta,X_i])\\	
&=&-B(V_\alpha,\sum_{k=1}^n x_{i\beta}^\gamma V_\gamma)\\	
&=& \Lambda_t \cdot x_{i\beta}^\alpha.
\end{eqnarray*}
Since $B([V_\alpha,V_\beta],X_i) +B([V_\beta,V_\alpha],X_i)=0$, the above steps show that for all $1\le i\le n$ we obtain 
$$\Lambda_t \cdot (x_{i\alpha}^\beta + x_{i\beta}^\alpha)=0.$$ 
This implies that $\mathcal{F}$ is totally geodesic. 	
\end{proof}

\begin{example}
Let $(G,g)$ be a six-dimensional Riemannian Lie group such that the orthonormal basis $\{A,B,C,X,Y,Z\}$ for its Lie algebra $\g$ satisfies the following Lie bracket relations
$$\lb AB = 2C,\quad \lb CA = 2B,\quad \lb BC = 2A,$$
\begin{eqnarray*}
&&\lb AX = -B,\quad \lb AY= -B,\quad\lb AZ = -B,\\
&&\lb BX = A,\quad \lb BY = A - 2C,\quad\lb BZ = A - 2C,\\
&&\lb CX = 0,\quad\lb CY = 2B,\quad \lb CZ = 2B,\\
&&\lb XY = B,\quad \lb XZ = B,\quad \lb YZ = 0.
\end{eqnarray*}
Here the compact semisimple $K=\SU 2$ is equipped with a biinvariant Riemannian metric (Example \ref{example-9}).  It is normal in $G$ since its Lie algebra $\k$, generated by $A,B,C$ is an ideal in $\g$.  The radical $\text{rad}(\g)$ is three dimensional so $G$ is not semisimple.  The left-invariant foliation $\F$ generated by $K$ is Riemannian and totally geodesic.  The horizontal distribution $\H$ is not integrable.   Employing the ideas of Remark \ref{remark-O'Neill} an easy calculations shows that the quotient group $(G/K,h)$ is the flat abelian $\rn^3$ (Example \ref{example-1}).
\end{example}

The proof of Theorem \ref{theorem-minimal} is rather straightforward, the first step is to show,  in Proposition \ref{Proposition-CK-minimal}, that there exists a Riemannian metric on any semisimple $K$ such that $\mathcal{F}$ is minimal.  
\smallskip

Let $\theta:\k\to\k$ be a {\it Cartan involution} on the semisimple subalgebra $\k$ of $\g$ i.e. an involutive automorphism, such that the bilinear form $g_K:\k\times\k\to\rn$ with 
$$g_K:(V,W)\mapsto -B(V,\theta(W))$$ 
is positive definite on $\k$.  Such an involution is unique up to inner automorphisms, see p.185 of \cite{Hel} for details.  This induces the {\it Cartan-Killing metric} on $K$ which we also denote by $g_K$.
\smallskip

\begin{proposition} 
\label{Proposition-CK-minimal}
Let $(G,g)$ be a Riemannian Lie group and $K$ be a semisimple normal subgroup of $G$ generating a left-invariant conformal foliation $\F$ on $G$.  If the metric $g$ restricted to $K$ is the Cartan-Killing metric $g_K$ then the foliation $\mathcal{F}$ is minimal.
\end{proposition}

\begin{proof}
Since $g|_K$ is the Cartan-Killing metric it follows that  $B(V_\alpha,V_\beta) = \delta_{\alpha \beta}\theta_\alpha$ for an orthonormal basis $\{V_1,\dots V_d\}$ of $\k$.  Here we denote $\theta(V_\alpha)$ by  $\theta_\alpha$.  Then for any horizontal basis vector $X_i$ we have
\begin{eqnarray*}
B([V_\alpha,V_\beta],X_i) 
&=& B(V_\alpha,[V_\beta,X_i])\\	
&=&-B(V_\alpha,\sum_{k=1}^n x_{i\beta}^\gamma V_\gamma)\\	
&=& \theta_{\alpha}\, x_{i\beta}^\alpha.
\end{eqnarray*}
Since $B([V_\alpha,V_\beta],X_i) = -B([V_\beta,V_\alpha],X_i)$, the above steps show that 
$$\theta_{\alpha}\, x_{i\beta}^\alpha = - \theta_{\beta}\, x_{i\alpha}^\beta.$$ 
Then for $\alpha=\beta$ we get that $x_{i\alpha}^\alpha=-x_{i\alpha}^\alpha = 0$. Thus the foliation $\F$ is minimal. 	
\end{proof}

\begin{example}
Let $(G,g)$ be a six-dimensional Riemannian Lie group such that the orthonormal basis $\{A,B,C,X,Y,Z\}$ for its Lie algebra $\g$ satisfies the following Lie bracket relations
$$\lb AB = 2C,\quad \lb CA = 2B,\quad \lb BC = -2A,$$
\begin{eqnarray*}
&&\lb AX=B + C,\quad \lb AY = B + C,\quad\lb AZ = B + C,\\ 
&&\lb BX=A - 2C,\quad\lb BY = A-4C,\quad\lb BZ = A-2C,\\
&&\lb CX=A + 2B,\quad \lb CY=A + 4B,\quad\lb CX=A + 2B,\\
&&\lb XY=-2A - 2B + 2Z,\quad \lb XZ = 4A + B - C - 2Y,\\
&&\lb YZ= -2A + 2C + 2X.
\end{eqnarray*}
Here the non-compact semisimple $K=\SLR 2$ (Example \ref{example-8}) is equipped with the Cartan-Killing metric.  It is normal in $G$ since its Lie algebra $\k$, generated by $A,B,C$ is an ideal in $\g$.
The radical $\text{rad}(\g)$ vanishes so $G$ is semisimple.  The left-invariant foliation $\F$ generated by $K$ is Riemannian and minimal but not totally geodesic.  The horizontal distribution $\H$ is not integrable.  The quotient group $G/K$ is the compact semisimple $\SU 2$. Using Remark \ref{remark-O'Neill} it is easy to show that the Riemannian Lie group $(G/K,h)$ has constant sectional curvature $+1$. 
\end{example}

We are now ready to prove our main result stated in Theorem \ref{theorem-minimal}.

\begin{proof}[Proof of Theorem \ref{theorem-minimal}]
Let $\{V_1,...,V_d,X_1,...,X_n\}$ be an orthonormal basis for $\mathfrak{g}$ with respect to the metric $g$ such that $V_1,...,V_n$ generate the subalgebra $\k$. 
Then notice that
$$\trace (\ad_{X_i}) = \sum_{\gamma=1}^d x_{i\gamma}^\gamma + \sum_{k=1}^n g(\ad_{X_i} (X_k),X_k).$$
Now equip $G$ with an additional left-invariant metric $\hat g$, which we can fully describe by its action on $\mathfrak{g}$.
First define $\hat g|_{\V \times\V}$ to be the Cartan-Killing metric on $K$ and then let
$$
\hat g|_{\H \times \H} = g|_{\H \times \H}, \ \ \hat g|_{\V \times \H} = g|_{\V \times \H}.
$$
Then we can use the Gram-Schmidt process to obtain an orthonormal basis $\{W_1,...,W_d,X_1,...X_n\}$ with respect to the metric $\hat g$. Since both metrics are left-invariant changing the metric simply amounts to a change of basis. Then it follows from Proposition \ref{Proposition-CK-minimal} that the structure constants with respect to the new metric $\hat g$,  
$$
\hat x_{i\alpha}^\alpha=\hat g(W_\alpha,[X_i,W_\alpha]),
$$ 
are all equal to zero. Then since the trace of an operator is invariant under the change of basis, and the fact that each $X_i$ is unchanged by the above Gram-Schmidt process, we get that
\begin{eqnarray*}
\sum_{\alpha=1}^d  x_{i\alpha}^\alpha  
&=& \trace (\ad_{X_i}) - \sum_{j=1}^n g(\ad_{X_i}(X_j),X_j)\\
&=& \trace (\ad_{X_i}) - \sum_{j=1}^n \hat g(\ad_{X_i}(X_j),X_j)\\
&=& \sum_{\alpha=1}^d \hat x_{i\alpha}^\alpha \\
&=& 0.
\end{eqnarray*}
Since this is true for all $1\le i\le n$ it follows that $\F$ is minimal.
\end{proof}

The following example is interesting in connection with Theorem \ref{theorem-minimal}.  The group $K$ is normal but its ambient group $G_\alpha$ is not semisimple.

\begin{example}
For $\alpha\in\rn^+$, let $(G_\alpha,g_\alpha)$ be a six-dimensional Riemannian Lie group such that the orthonormal basis $\{A,B,C,X,Y,Z\}$ for its Lie algebra $\g$ satisfies the following Lie bracket relations
$$[A, B] = C,\quad [C, A] = B,\quad [B, C] = 4A,$$
\begin{eqnarray*}
&&[A, X] = B + C,\quad [A, Y] = 0,\quad [A, Z] = B + C,\\
&&[B, X] = -4A - C,\quad [B, Y] = -C,\quad [B, Z] = -4A - C,\\
&&[C, X] = -4A + B,\quad [C, Y] = B,\quad [C, Z] = -4A + B,\\
&&[X, Y] = -B - C,\quad  [Y, Z] = -A + B + C + Y,\\
&&[X, Z] = \alpha A + \alpha B - \alpha\, C - \alpha X.
\end{eqnarray*}
The semisimple subgroup $K$ is normal in $G_\alpha $ since its Lie algebra $\k$, generated by $A,B,C$ is an ideal in $\g_\alpha$. The quotient $G_\alpha/K$ is the solvable Lie group $\text{Sol}_\alpha^3$ (Example \ref{example-6}). The radical $\text{rad}(\g_\alpha)$ is three-dimensional so $G_\alpha$ is not semisimple.  The left-invariant foliation $\F_\alpha$ generated by $K$ is Riemannian and minimal but not totally geodesic.  The horizontal distribution $\H_\alpha$ is not integrable.
\end{example}

We remind the reader of the fact the three-dimensional Lie groups $\text{Sol}_\alpha^3$ do not carry conformal foliations with minimal leaves, independent of which left-invariant metrics they are given.  The same applies to the Lie group $G_4$ in Example \ref{example-4}.  For this see Theorem 4.6 of \cite{Gud-Sve-5}.

\section{Normality for Codimension Three}
\label{section-normality-dimension-3}

The main aim of this section is to prove Proposition \ref{proposition-not-perfect}.  

\begin{definition}
A Lie group $G$ is said to be {\it perfect} if its Lie algebra $\g$ satisfies $[\g,\g]=\g$.
\end{definition}

It is well-known that every semisimple Lie group is perfect. The following example shows that the converse is not true.
\smallskip

\begin{example}
Let $(G,g)$ be a six-dimensional Riemannian Lie group such that the orthonormal basis $\{A,B,C,X,Y,Z\}$ for its Lie algebra $\g$ satisfies the following Lie bracket relations
$$[A, B] = C,\quad [C, A] = B,\quad [B, C] = A,$$
\begin{eqnarray*}
\lb AX&=& -A-C-Z,\quad \lb AY= -B-C,\\
\lb AZ&=& A + B + X,\quad \lb BX= C,\\
\lb BY&=&2A + C + Z,\quad \lb BZ= -A - B + 2C - Y,\\
\lb CX&=& - 2B + C - Y,\quad  \lb CY= 2A + X,\\ 
\lb CZ&=& - B,\quad \lb XY= Y - 2Z - 2A + 2B - C,\\
\lb XZ&=& -X + Y - Z - A - 2C,\quad \lb YZ= -X + Y - B - 2C.
\end{eqnarray*}
\end{example}
The Lie algebra $\g$ is perfect and its radical $\text{rad}(\g)$ is generated by the three elements
$A+Z$, $B+X-Z$ and $C+X-Y$.  Hence $G$ is not semisimple. 
The compact semisimple subgroup $K=\SU 2$ is not normal since its Lie algebra $\k$ generated by $A,B,C$ is not an ideal in $\g$. The left-invariant foliation $\F$ generated by $K$ is Riemannian and minimal but not totally geodesic.  The horizontal distribution $\H$ is not integrable.

\begin{proposition}
\label{proposition-not-perfect}
Let $(G,g)$ be a Riemannian Lie group with semisimple subgroup $K$ generating a left-invariant conformal foliation $\F$ on $G$ of codimension three. If $G$ is not perfect then $K$ is normal in $G$.
\end{proposition}

\begin{proof}
Let $\{V_1, \dots V_d,X,Y,Z\}$ be an orthonormal basis for the Lie algebra $\mathfrak{g}$ of $G$ such that $X,Y,Z$ generate the three-dimensional horizontal distribution $\H$.  According to Proposition \ref{proposition-conformal-foliation-is-Riemannian} the foliation $\F$ is Riemannian. In particular this means that
\[
  \begin{aligned}[t]
 \H [X,V_\alpha] &=&  & & &r_\alpha \cdot Y& + s_\alpha \cdot Z, \\
 \H [Y,V_\alpha] &=&  &- r_\alpha \cdot X& & &  +  t_\alpha \cdot Z, \\
 \H [Z,V_\alpha] &=& &- s_\alpha \cdot X&  &-  t_\alpha \cdot Y&   .
  \end{aligned}
\]
It is clear that $K$ is normal in $G$ if and only if $ r_\alpha, s_\alpha$, $t_\alpha$,  all vanish for each $\alpha$. Suppose this is not the case, then by reordering the $V_\alpha$, we can suppose that at least one of $r_1,s_1,t_1$ is non-zero.

Since $\g$ is not perfect we have that
 $$[\mathfrak{g},\mathfrak{g}] \neq \mathfrak{g}.$$
In particular $[\mathfrak{g},\mathfrak{g}]^\perp$ is a non-trivial subspace. Since $[\V,\V] = \V \subset [\g,\g]$ it follows that $[\mathfrak{g},\mathfrak{g}]^\perp$ is contained in the horizontal subspace $\H$.

By the definition of the derived subalgebra we have that $$\V+\H [V_i,X]+\H [V_i,Y]+\H [V_i,Z] \subset [\g,\g]$$  
and so $\H[\g,\g]$ contains the span of $\{\H [V_i,X],\H [V_i,Y],\H [V_i,Z]\}$. An arbitrary linear combination of these vectors is given by
\begin{eqnarray*}
 \lambda( r_\alpha \cdot Y + s_\alpha \cdot Z)+ \mu(- r_\alpha \cdot X+ t_\alpha \cdot Z)+ \eta (- s_\alpha \cdot X -  t_\alpha \cdot Y) \\
 = -(\mu  r_\alpha+\eta s_\alpha) \cdot X + (\lambda  r_\alpha-\eta t_\alpha) \cdot Y+(\lambda s_\alpha +\mu  t_\alpha) \cdot Z,
\end{eqnarray*}
which can be written as
$$
\begin{bmatrix}
- r_\alpha &  s_\alpha & 0 \\
0 &   t_\alpha &  r_\alpha \\
 t_\alpha & 0 &  s_\alpha
\end{bmatrix}\cdot 
\begin{pmatrix}
\mu\\ -\eta\\ \lambda
\end{pmatrix}.
$$
If $ r_\alpha, s_\alpha, t_\alpha$ are all non-zero for some $\alpha$, then this matrix has maximum rank and $\mathfrak{g}$ would be perfect, so at least one must be zero. 
Without loss of generality we assume that $r_1=0$. Now we show that with this assumption $ r_\alpha=0$ for all $\alpha = 1,\dots d$. Suppose that  $ r_\beta \neq 0$ and, without loss of generality, that $ s_\beta= 0$. Then we have the following structure equations
\begin{eqnarray*}
\H [X,V_1] &=& s_1 \cdot Z, \\
\H [Y,V_1]&=& t_1 \cdot Z, \\
\H [Z,V_1] &=& -s_1 \cdot X - t_1 \cdot Y.
\end{eqnarray*}
So, since we are assuming that at least one of $s_1$ and $t_1$ is non-zero it follows that $Z \in [\mathfrak{g},\mathfrak{g}]$. Similarly we compute for $\beta$:
\begin{eqnarray*}
\H [X,V_\beta] &=&r_\beta \cdot Y, \\
\H [Y,V_\beta]&=& -r_\beta \cdot X  + t_\beta \cdot Z, \\
\H [Z,V_\beta] &=&  - t_\beta \cdot Y.
\end{eqnarray*}
Similarly we see that $Y \in [\mathfrak{g},\mathfrak{g}]$, and finally since $r_\beta \cdot X  +t_\beta \cdot Z \in  [\mathfrak{g},\mathfrak{g}]$ we can conclude that $X$ must be also, so $G$ would be perfect. Thus, if $r_1=0$ we have that  $r_\alpha=0$ for all $\alpha$. \newline
Now we will show that $s_\alpha$ and $t_\alpha$ must also vanish. We recall the Jacobi identity
$$ 0= [ V_\alpha, [V_\beta,X]] + [V_\beta, [X,V_\alpha]]+[X,[V_\alpha,V_\beta]].$$
Furthermore, since $K$ is semisimple, there exists a linear combination $$ \sum_{\alpha<\beta} a_{\alpha \beta}^\gamma[V_\alpha,V_\beta]= V_\gamma$$ for any $\gamma =1,\dots d$. 
Since $ r_\alpha=0$ for all $\alpha$ we have, for any $\alpha,\beta$
\begin{eqnarray*}
 g([ V_\alpha, [V_\beta,X]], Z ) &=& g( [ V_\alpha,  - s_\beta \cdot Z ], Z ) \\
 &=& s_\beta \cdot g( -s_\alpha \cdot X - t_\alpha \cdot Y, Z ) \\
 &=& 0.
\end{eqnarray*}

Finally we compute

\begin{eqnarray*}
0&=& g(\sum_{\alpha<\beta} a_{\alpha \beta}^\gamma \left( [ V_\alpha, [V_\beta,X]] + [V_\beta, [X,V_\alpha]]+[X,[V_\alpha,V_\beta]] \right), Z ) \\
&=& \sum_{\alpha<\beta} \left( a_{\alpha \beta}^\gamma \cdot g(  \left( [ V_\alpha, [V_\beta,X]] + [V_\beta, [X,V_\alpha]] \right), Z ) \right)\\
&&\qquad\qquad \qquad\qquad \qquad\qquad 
+ g(  \sum_{\alpha<\beta} a_{\alpha \beta}^l[X,[V_\alpha,V_\beta]] , Z ) \\
&=& 0+g( [X, V_\gamma], Z ) \\
&=& s_\gamma
\end{eqnarray*}
and so $s_\gamma=0$ for all $\gamma$. Replacing $X$ with $Y$ in the above computation shows that $t_\gamma = 0$, and hence $K$ is normal in $G$.
\end{proof}

Next we provide an example that shows that the result of Proposition \ref{proposition-not-perfect} does not apply in the case when the codimension of $K$ in $G$ is four.

 \begin{example}
Let $(G,g)$ be a seven-dimensional Riemannian Lie group such that the orthonormal basis $\{A,B,C,X,Y,Z,W\}$ for its Lie algebra $\g$ satisfies the following bracket relations
$$[A, B] = 2C,\quad [C, A] = 2B,\quad [B, C] = 2A,$$
\begin{eqnarray*}
&&[A, X] = 2Z,\quad [A, Y] = 0,\quad [A, Z] = -2X,\quad [A, W] = 0,\\ 
&&[B, X] = 0,\quad [B, Y] = 2Z,\quad [B, Z] = -2Y,\quad [B, W] = 0,\\
&&[C, X] = 2Y,\quad [C, Y] = -2X,\quad [C, Z] = 0,\quad [C, W] = 0,\\ 
&&[X, Y] = C,\quad [X, Z] = A,\quad [Y, Z] = B.
\end{eqnarray*}
The Lie algebra $\g$ is not perfect so it is not semisimple. The compact subgroup $K=\SU 2$ is not normal since its Lie algebra $\k$ generated by $X,Y,Z$ is not an ideal in $\g$. The left-invariant foliation $\F$ generated by $K$ is Riemannian and totally geodesic.  The horizontal distribution $\H$ is not integrable.  The quotient space $(G/K,h)$ is the flat $\rn^4$.
\end{example}

\section{Acknowledgments}
\label{section-acknowledgments}

The authors would like to thank the anonymous referees for their useful comments leading to an improved presentation of this work.

\appendix
\section{The Lie Algebras of Dimension Two and Three}
\label{section-Bianchi}

The two-dimensional Lie algebras fall into two isomorphy classes, the abelian and the non-abelian types.

\begin{example}[Type A]
\label{example-01}
The two-dimensional abelian Lie algebra we denote by $\r_2$.  The corresponding simply connected Lie group is the abelian  $\rn^2$ which we equip with its standard flat Euclidean metric.
\end{example}

\begin{example}[Type B] 
\label{example-02}
Let $\h_2$ denote the Lie algebra with an orthonormal basis $\{X,Y\}$ satisfying
$$[X,Y]=\rho\cdot Y.$$
The corresponding simply connected Lie group is the hyperbolic plane $H^2$ of constant curvature $-\rho^2$.
\end{example}

At the end of the 19th century Luigi Bianchi classified the $3$-dimensional real Lie algebras.  They fall into nine disjoint types I-IX. Each contains a single isomorphy class except types VI and VII which are continuous families of different classes.  For reference we list below Bianchi's classification and notation for the corresponding simply connected Lie groups. We also equip these Lie groups with the left-invariant Riemannian metric for which the given basis $\{X,Y,Z\}$ of each Lie
algebra is orthonormal.

\begin{example}[Type I]
\label{example-1}
The three-dimensional Lie algebra $\r_3$ is abelian and the corresponding simply connected Lie group is the abelian $\rn^3$ which we equip with its standard flat Euclidean metric.
\end{example}

\begin{example}[Type II] 
\label{example-2}
The nilpotent Lie algebra $\mathfrak{nil}_3$ has basis $\{X,Y,Z\}$ satisfying
$$[X,Y]=Z.$$
The corresponding simply connected Lie group is the nilpotent Heisenberg group $\text{Nil}^3$.
\end{example}

\begin{example}[Type III]
\label{example-3}
The Lie algebra $\h^2\oplus\rn$ is the direct sum of $\h_2$ and the one-dimensional $\r_1$ given by 
$$[Y,X]=X.$$
The corresponding simply connected Lie group is denoted by $H^2\times\rn$. Here $H^2$ is the standard hyperbolic plane.
\end{example}

\begin{example}[Type IV]
\label{example-4}
The Lie algebra $\g_4$ with basis $\{X,Y,Z\}$ satisfying
$$[Z,X]=X,\quad [Z,Y]=X+Y.$$
The corresponding simply connected Lie group is denoted by $G_4$.
\end{example}

\begin{example}[Type V]
\label{example-5}
The Lie algebra $\h^3$ with a basis $\{X,Y,Z\}$ satisfying
$$[Z,X]=X,\quad [Z,Y]=Y.$$
The corresponding simply connected Lie group $H^3$ is the standard three dimensional hyperbolic space of constant sectional curvature $-1$
\end{example}

\begin{example}[Type VI] 
\label{example-6}
For $\alpha\in\rn^+$, the Lie algebra $\sol_\alpha^3$ has basis $\{X,Y,Z\}$ satisfying
$$[Z,X]=\alpha X,\quad [Z,Y]=-Y.$$
The corresponding simply connected Lie group is denoted by
$\text{Sol}_\alpha^3$.
\end{example}

\begin{example}[Type VII] 
\label{example-7}
The Lie algebra $\g_7(\alpha)$, where
$\alpha\in\rn$, is the Lie algebra with basis $\{X,Y,Z\}$ satisfying
$$[Z,X]=\alpha X-Y,\quad [Z,Y]=X+\alpha Y.$$
The corresponding simply connected Lie group
is denoted by $G_7(\alpha)$.
\end{example}

\begin{example}[Type VIII] 
\label{example-8}
The Lie algebra $\slr 2$ has basis $\{X,Y,Z\}$
satisfying
$$[X,Y]=-2Z,\quad [Z,X]=2Y,\quad [Y,Z]=2X.$$
The corresponding simply connected Lie group is denoted by
$\widetilde{\SLR 2}$ as it is the universal cover of the special
linear group $\SLR 2$.
\end{example}

\begin{example}[Type IX] 
\label{example-9}
The Lie algebra $\su 2$ has basis $\{X,Y,Z\}$ satisfying
$$[X,Y]=2Z,\quad [Z,X]=2Y,\quad [Y,Z]=2X.$$
The corresponding simply connected Lie group is  $\SU 2$. This is isometric to the standard three-dimensional sphere $S^3$ of constant curvature $+1$. 
\end{example}


                 

\end{document}